\documentclass[a4paper]{llncs}

\usepackage{tikz-cd} 

\usepackage{relsize}
\usepackage{cite}
\usepackage{color}
\usepackage{hyperref}
\hypersetup{
	colorlinks   = true, 
	urlcolor     = green, 
	linkcolor    = black,
	citecolor   = red 
}
\usepackage{amsmath,amssymb}

\usepackage{mathtools}
\usepackage[margin=2.6cm]{geometry}

\makeatletter
\usepackage{comment}
\let\wfs@comment@comment\comment
\let\comment\@undefined

\usepackage{changes}
\let\wfs@changes@comment\comment
\let\comment\@undefined

\newcommand\comment{%
	\ifthenelse{\equal{\@currenvir}{comment}}
	{\wfs@comment@comment}
	{\wfs@changes@comment}%
}

\definechangesauthor[name=Matteo, color=red]{MAT}
\definechangesauthor[name=Martino, color=blue]{MAR}
\definechangesauthor[name=Eimear, color=violet]{EIM}
\usepackage{stmaryrd}
\newcommand{\numberset}{\mathbb}

\newcommand{\F}{\numberset{F}}

\newcommand{\C}{\mathcal{C}}
\newcommand{\mS}{\mathcal{S}}

\newcommand{\mC}{\mathcal{C}}

\newcommand{\mN}{\mathcal{N}}

\newcommand{\mU}{\mathcal{U}}

\newcommand{\mH}{\mathcal{H}}

\newcommand{\Fq}{\F_q}

\newcommand{\bs}{\bf{s}}

\newcommand\qbin[3]{\left[\begin{matrix} #1 \\ #2 \end{matrix} \right]_{#3}}

\newcommand{\bu}{\mathbf{u}}
\newcommand{\bv}{\mathbf{v}}

\newcommand{\bn}{\mathbf{n}}

\newcommand{\rk}{\textnormal{rk}}

\newcommand{\srk}{\textnormal{srk}}

\DeclareMathOperator{\supp}{supp}

\DeclareMathOperator{\PG}{PG}

\newcommand{\Fqm}{\mathbb{F}_{q^m}}

\providecommand{\keywords}[1]
{
	
	\noindent\textbf{{\small \textbf{Keywords.}}} #1
}

\usepackage{array}
\newcolumntype{x}[1]{>{\centering\arraybackslash\hspace{0pt}}p{#1}}
\newcolumntype{y}[1]{>{\centering\arraybackslash\hspace{0pt}}m{#1}}

\usepackage{tikz,pgfplots}
\pgfplotsset{compat=newest}
\usepackage{setspace}

\title{The geometry of covering codes in the sum-rank metric}

\author{Matteo Bonini \inst{1} \and Martino Borello \inst{2,3} \and Eimear Byrne\inst{4}}
\institute{Aalborg University, Department of Mathematical Sciences, Aalborg, Denmark\
	\email{mabo@math.aau.dk}
	\and
	Universit\'e Paris 8, Laboratoire de G\'eom\'etrie, Analyse et Applications, LAGA,
	Universit\'e Sorbonne Paris Nord, CNRS, UMR 7539, France\
	\email{martino.borello@univ-paris8.fr}
	\and 
	INRIA,  France
	\and
	School of Mathematics and Statistics, University College Dublin, Ireland\
	\email{ebyrne@ucd.ie}}

\date{}

\begin{document}
	
	\maketitle
	
	\begin{abstract}
		We introduce the concept of a sum-rank saturating system and outline its correspondence to covering properties of a sum-rank metric code. We consider the problem of determining the shortest length of a sum-rank-$\rho$-saturating system of a fixed dimension, which is equivalent to the covering problem in the sum-rank metric. We
		obtain upper and lower bounds on this quantity.
		We also give constructions of saturating systems arising from geometrical structures.
	\end{abstract}
	
	
	\section*{Introduction}
	Researchers have extensively explored the connections between linear codes and sets of points in finite geometries, as evidenced by previous works such as \cite{calderbank1986geometry,dodunekov1998codes,alfarano2019geometric,davydov1995constructions,davydov2000saturating,davydov2011linear}. 
	Generator or parity check matrices of a linear code can be identified with multisets 
	of projective points, while the supports of codewords can be identified with complements of hyperplanes in a fixed projective set. The interconnection between these two domains facilitates the application of methods from one field to the other. Notably, this approach has been used to construct codes with a bounded {\em covering radius}, associated with {\em saturating sets} in projective space.
	Recent investigations in the geometry of rank-metric  codes\cite{alfarano2021linear,randrianarisoa} reveal their correspondence to $q$-systems and linear sets. A similar correspondence holds for sum-rank metric codes \cite{santonastaso2022subspace, neri2023geometry}.
	
	The sum-rank metric may be viewed as a hybrid of the rank and Hamming metric, extending both notions. Sum-rank metric codes have found applications in network coding, space-time coding, and distributed storage systems \cite{Nbrega2010MultishotCF,M-PKsr,M-PKlr,M-PSK}. They have been studied in terms of decoding and code-optimality (see, e.g. \cite{abiad+,chen,sven+}). One of the key reasons this metric has gained attention in recent years is that sum-rank metric codes beat traditional codes in terms of the field size needed to construct codes that meet the Singleton-like bound, due to the existence of linearized Reed-Solomon codes \cite{martinez2018skew}.  In this paper, we focus on sum-rank metric codes that are $\F_{q^m}$-linear subspaces. While there exists a more general description of sum-rank metric codes simply as linear spaces of matrices over $\Fq$, this restriction has immediate connections to geometric approaches \cite{neri2023geometry,borello2023geometric}. 
	
	This paper focuses on the covering radius.  The covering radius of a code is the smallest positive integer $\rho$ such that the union of the spheres of radius $\rho$ about each codeword is equal to the entire ambient space.  The covering radius serves as an indicator of combinatorial properties, such as {\em maximality}, and is an invariant of code equivalence. It also provides insight into error-correcting capabilities by determining the maximal weight of a correctable error.
	This essential coding theoretical parameter has been extensively studied for codes in the context of the Hamming metric \cite{brualdiplesswil,coveringcodes,davydov1995constructions,davydov2011linear,davydov2000saturating,denaux2021constructing,davydov2003saturating}. However, only a few papers in the literature on rank-metric codes and sum-rank metric codes address this parameter \cite{byrne2017covering,gadouleauphd, ott, bonini2022saturating}.
	Recently, in \cite{bonini2022saturating}, a purely geometrical approach based on saturating systems was proposed to study the covering radius in the rank metric. This approach led to new bounds and interesting examples of covering codes in the rank metric, see \cite{bonini2022saturating, bartoli2023saturating}. 
	
	In this paper, we extend these ideas to the sum-rank metric by introducing the concept of a sum-rank saturating system, aligning it with a sum-rank metric covering code. We also provide new bounds for covering codes in the sum-rank metric, as well as examples arising from partitions and cutting systems.
	
	After recalling some main definitions and results in Section \ref{sec:Prel}, we introduce the main object of the paper in Section \ref{sec:SRsatsyst}: we introduce the notion of a sum-rank saturating system, give equivalent characterizations of such systems, and outline the connection to the rank covering radius of a sum-rank metric code. In Section \ref{sec:bounds} we give upper and lower bounds on the minimum $\F_q$-dimension of a sum-rank saturating system. Finally, in \ref{sec:constructions} we provide explicit constructions of sum-rank saturating systems from partitions of projective spaces and cutting designs.
	
	\section{Preliminaries}   \label{sec:Prel}
	
	\subsection{Vector sum-rank metric codes} 
	Throughout this paper, we will let $t$ denote a fixed positive integer and we will by
	$\mathbf{n}=(n_1,\ldots,n_t) \in \mathbb{N}^t$ an ordered tuple with $n_1 \geq n_2 \geq \ldots \geq n_t$ and $N = n_1+\ldots+n_t$. We use the notation $\F_{q^m}^\mathbf{n}=\bigoplus_{i=1}^t\F_{q^m}^{n_i}$ for the direct sum of vector spaces $\F_{q^m}^{n_i}$.
	We write $\mN$ to denote the set $\{(a_1,\ldots,a_t): 0 \leq a_i \leq n_i\}$.

	\begin{definition}
		Given a pair of nonnegative integers $n$ and $m$, the $q$-\textbf{binomial} or \textbf{Gaussian coefficient} counts the number of $m$-dimensional subspaces of an $n$-dimensional subspace over $\F_q$ and is defined to be:
		$$\qbin{n}{m}{q}:=
		\left\{
		\begin{array}{cl}
			\displaystyle \prod_{i=0}^{m-1}\frac{q^n-q^i}{q^m-q^i}   &  \text{ if } n \geq m > 0,\\
			0 & \text{ if } m > n, \\
			1 & \text{ if } m=0.\\
		\end{array}\right.$$
	\end{definition}
	
	We will adopt the following notation: for $\bu=(u_1,\dots,u_t),\bv=(v_1,\cdots,v_t)$ we define,
	\begin{align*}
		|\bu|&:= \sum_{j=1}^t u_j,&
		\qbin{\bu}{\bv}{q}&:=\prod_{j=1}^t  \qbin{u_j}{v_j}{q},&
		q^{\binom{\bu}{2}} &:=\prod_{j=1}^t q^{\binom{u_j-v_j}{2}}. 
	\end{align*}

	
	We recall the following definitions.
	
	\begin{definition}
		The  rank or rank weight of a vector $v=(v_1,\ldots,v_n) \in \F_{q^m}^n$ is defined to be: 
		$$\textup{w}_\rk(v):=\rk(v):=\dim_{\F_q} (\langle v_1,\ldots, v_n\rangle_{\F_q}).$$
		The sum-rank weight of an element $x=(x_1 ,\ldots, x_t) \in \F_{q^m}^\mathbf{n}$ is 
		defined to be $\displaystyle\textup{w}_{\textup{srk}}(x):=\sum_{i=1}^t \rk(x_i).$
		The sum-rank distance between $x=(x_1 , \ldots , x_t)$ and $ y=(y_1 , \ldots, y_t)$ in 
		$\F_{q^m}^\mathbf{n}$ is defined to be:
		\[
		d(x,y):=w(x-y)=\sum_{i=1}^t \rk(x_i-y_i).
		\]
		We say an $\F_{q^m}$-subspace $\C$ of $\F_{q^m}^{\mathbf{n}}$ is an $[\mathbf{n},k,d]_{q^m/q}$ sum-rank metric code (or an $[\mathbf{n},k]_{q^m/q}$ (sum-rank metric) code) if $k$ is the $\F_{q^m}$-dimension of $\C$ and $d$ is its minimum distance, with respect to the sum-rank metric, i.e.,
		\[
		d=d(\C):=\min\{d(x,y) \colon x, y \in \C, x \neq y  \}.
		\]    
		The dual of $\mC$ is defined to be
		$\C^\perp :=\{ v \in \F_{q^m}^{\mathbf{n}}:\: v \cdot c = 0\: \forall\: c \in \C \}$.
		A matrix $G=(G_1\lvert \ldots \lvert G_t) \in \F_{q^m}^N$ with $G_1,\ldots,G_t \in \F_{q^m}^{k \times {n_i}}$ is called a generator matrix of $\C$ if $G$ has rank $k$ and its $\F_{q^m}$-row-space is $\C$.
		A parity-check matrix of $\C$ is a generator matrix of $\C^\perp$.
		We say that $\C$ is nondegenerate if the columns of $G_i$ are $\F_q$-linearly independent for $i\in \{1,\ldots,t\}$. 
	\end{definition}
	Without loss of generality, in this paper we will only consider nondegenerate codes.
	
	\begin{definition}
		The sum-rank metric covering radius of $\mC\subset \F_{q^m}^{\mathbf{n}}$ is defined to be:
		$$\rho_{\srk}(\mC):= \max \{ \min \{ d(x,c) : c \in \C\} : x \in\F_{q^m}^{\mathbf{n}}\}.$$
	\end{definition}
	
	Let $H=(H_1|\cdots|H_t) \in \F_{q^m}^{\mathbf{n}}$ be a parity-check matrix of a sum-rank metric code $\mC\subset \F_{q^m}^{\mathbf{n}}$. Then 
	$\rho_{\srk}(\mC)=\rho$, if and only if for every
	$v\in \F_{q^m}^{n-k}$ there exists $\lambda=(\lambda_1,\ldots,\lambda_t)\in \F_{q^m}^{1\times n_1}\times \cdots\times \F_{q^m}^{1\times n_t}$ with 
	$\textup{w}_{\textup{srk}}(\lambda)\le \rho$ such that $$v=H(\lambda_1,\ldots,\lambda_t)^T,$$ and $\rho$ is the least integer with this property.
	
	

	\subsection{$q$-Systems}
	We introduce some results regarding the connections of sum-rank metric codes and sets of subspaces, see \cite{neri2023geometry} for further details.

	\begin{definition}
		For each $i\in \{1,\dots,t\}$, let $\mathcal{U}_i$ be an $\F_q$-subspace of $\F_{q^m}^k$ of dimension $n_i$. 
		If the ordered $t$-tuple $\mathcal{U}=(\mathcal{U}_1,\ldots,\mathcal{U}_t)$ satisfies
		$ \langle \mathcal{U}_1, \ldots, \mathcal{U}_t \rangle_{\F_{q^m}}=\F_{q^m}^k$
		then $\mU$ is called an \emph{$[\bn,k]_{q^m/q}$ system}.
		We say that $\mU$ has dimension (or rank) ${\bf n}$.  
		A \emph{generator matrix} for $\mU$ is a $k\times N$ matrix over $\F_{q^m}$ of the form $G = (G_1 | \cdots | G_t)$,
		where for each $i$, $G_i$ is a generator matrix for the $[n_i,k]_{q^m/q}$ system $\mU_i$, i.e., such that the $\Fq$-span of the columns of each $G_i$ is $\mU_i$.
	\end{definition}
	
	For any $\Fq$-subspace $\mathcal V$ of $\F_{q^m}^k$ and $a \in \F_{q^m}$ we write $a {\mathcal V} :=\{ av : v \in {\mathcal V} \}$.
	
	\begin{definition}
		Two sum-rank systems $(\mathcal{U}_1,\ldots,\mathcal{U}_t)$ and $(\mathcal{V}_1,\ldots, \mathcal{V}_t)$ are equivalent if there exists an isomorphism $\varphi\in \mathrm{GL}(k,\F_{q^m})$, an element $\mathbf{a}=(a_1,\ldots,a_t)\in (\F_{q^m}^*)^t$ and a permutation $\sigma\in\mathcal{S}_t$, such that for every $i\in\{1,\dots,t\}
		$
		$$ \varphi(\mathcal{U}_i) = a_i\mathcal{V}_{\sigma(i)}.$$
	\end{definition}

	The following result allows us to establish a connection between systems and codes that will be very useful later on; see \cite{neri2023geometry} for further details.
	
	\begin{theorem}[\!{\cite[Theorem 3.1]{neri2023geometry}}]
		Let $\C$ be an $[\mathbf{n},k,d]_{q^m/q}$. Let $G=(G_1\lvert \ldots \lvert G_t)$ be a generator matrix of $\C$.
		Let $\mU_i \subseteq \F_{q^m}^k$ be the $\F_q$-span of the columns of $G_i$, for $i\in \{1,\ldots,t\}$.
		The sum-rank weight of an element $x G \in \C$, with $x=(x_1,\ldots,x_k) \in \F_{q^m}^k$ is
		\begin{equation}\label{eq:weight}
			\textup{w}_{\textup{srk}}(x G) = N - \sum_{i=1}^t \dim_{\F_q}(\mU_i \cap x^{\perp}),
		\end{equation}
		where $x^{\perp}=\{y=(y_1,\ldots,y_k) \in \F_{q^m}^k \colon \sum_{i=1}^k x_iy_i=0\}$. In particular, the minimum distance of $\C$ is given by:
		\begin{equation} \label{eq:distancedesign}
			d=N- \max\left\{ \sum_{i=1}^t \dim_{\F_q}(\mU_i \cap H)  \colon H\mbox{ is an } \F_{q^m}\mbox{-hyperplane of }\F_{q^m}^k  \right\}.
		\end{equation}
		In particular, $(\mU_1,\ldots,\mU_t)$ in an $[\mathbf{n},k,d]_{q^m/q}$-system. 
	\end{theorem}
	
	Moreover, there is a one-to-one correspondence  between equivalence  classes  of  sum-rank nondegenerate $[\mathbf{n},k,d]_{q^m/q}$ codes and equivalence classes of $[\mathbf{n},k,d]_{q^m/q}$-systems; see \cite{neri2023geometry}.

	\subsection{Linear sets}
	
	Let us define linear sets, which were introduced by Lunardon in \cite{lunardon1999normal} for the construction of blocking sets, and which have become a topic of significant research in recent years. A thorough discussion of linear sets is available in \cite{polverino2010linear}.
	
	\begin{definition}
		Let $V$ be a $k$-dimensional vector space over $\F_{q^m}$ and consider $\Lambda=\PG(V,\F_{q^m})=\PG(k-1,q^m)$.
		Let $\mU$ be an $\F_q$-subspace of $V$ of dimension $n$. Then the point-set:
		\[ L_\mU=\{\langle { u} \rangle_{\mathbb{F}_{q^m}} : { u}\in \mU\setminus \{{ 0} \}\}\subseteq \Lambda, \]
		is called an $\F_q$-\textit{linear set of rank $n$}.
	\end{definition}
	\begin{definition}
		Let $P=\langle v\rangle_{\F_{q^m}}$ be a point in $\Lambda$. The \textit{weight of $P$ in $L_\mU$} is defined to be: 
		\[ w_{L_\mU}(P)=\dim_{\F_q}(\mU\cap \langle v\rangle_{\F_{q^m}}). \] 
	\end{definition}
	A basic upper bound on the number of points that a linear set contains is
	\begin{equation}\label{eq:card}
		|L_\mU| \leq \frac{q^n-1}{q-1}.
	\end{equation}
	We say that, $L_\mU$ is \textit{scattered} if it meets this number of points, or equivalently, if all points of $L_\mU$ have weight one.

	\section{Sum-rank saturating systems}\label{sec:SRsatsyst}
	In this section we will discuss the main object of this paper, namely, sum-rank saturating systems. 
	
	We start by recalling the definition of a $\rho$-saturating set.
	\begin{definition}
		Let $\mS \subseteq  \mathrm{PG}(k-1, q^m)$.
		\begin{itemize}
			\item[{\rm (a)}]
			A point $Q \in \mathrm{PG}(k-1, q^m)$ is said to be $\rho$-{\em saturated} by $\mS$ if there exist $\rho+1$ points  $P_1,\ldots,P_{\rho+1}\in \mathcal{S}$ such that $Q\in \langle P_1,\ldots,P_{\rho+1}\rangle_{\F_{q^m}}$. We also say that $\mS$ $\rho$-{\em saturates} $Q$.
			\item[{\rm (b)}]
			The set $\mS$ is called a $\rho$-\emph{saturating} set of $\mathrm{PG}(k-1, q^m)$ if every point $Q \in \mathrm{PG}(k-1, q^m)$ is $\rho$-saturated by $\mS$ and $\rho$ is the smallest value with this property.  
		\end{itemize}
	\end{definition}
	
	The following is the main object of this paper.
	
	\begin{definition}
		An $[\mathbf{n},k]_{q^m/q}$ system $\mU$ is called sum-rank-$\rho$-saturating if $L_{\mU_1}\cup\cdots\cup L_{\mU_t}$ is a $(\rho-1)$-saturating set.
	\end{definition}
	
	As in the rank-metric case, we may get a characterisation of sum-rank saturating systems.
	\begin{theorem}\label{th:tfaeabc}
		Let $\mU$ be an $[\mathbf{n},k]_{q^m/q}$ system and let $G$ be any generator matrix of $\mU$. 
		The following are equivalent:
		\begin{itemize}
			\item[{\rm (a)}]\label{th:tfaea} $\mathcal{U}$ is sum-rank-$\rho$-saturating.
			\item[{\rm (b)}]\label{th:tfaeb}  For each vector $v\in \F_{q^m}^k$ there exists $\lambda=(\lambda_1,\ldots,\lambda_t)\in \F_{q^m}^{1\times n_1}\times \ldots\times \F_{q^m}^{1\times n_t}$ with 
			$\textup{w}_{\textup{srk}}(\lambda)\le \rho$ such that $$v=G(\lambda_1,\ldots,\lambda_t)^T,$$ and $\rho$ is the least integer with this property.
			\item[{\rm (c)}]\label{th:tfaec}  We have
			\[\F_{q^m}^k=\bigcup_{\substack{(\mathcal{S}_i: i \in [t]) :\:\mathcal{S}_i 
					\leq_{\F_q}\mathcal{U}_i,\\\sum_{i=1}^t\dim_{\F_q}\mathcal{S}_i\leq \rho}}
			\left(\bigcup_{i=1}^t\langle \mathcal{S}_i \rangle_{\F_{q^m}}\right)
			\]
			and $\rho$ is the least integer with this property.
		\end{itemize}
	\end{theorem}
	\begin{proof}
		$(a) \implies (b)$ Let $Q = \langle v \rangle_{\F_{q^m}} \in PG(k-1,q^m)$. Since $\mathcal{U}$ is sum-rank-$\rho$-saturating, there exist $P_1,\ldots, P_\rho \in L_{\mU_1} \cup \cdots \cup L_{\mU_t}$
		such that $Q\in \langle P_1,\ldots,P_{\rho}\rangle_{\F_{q^m}}$.
		For each $i$, let $P_i=\langle w_i \rangle_{\F_{q^m}}$ for some $w_i \in \mU_{w(i)}$ and $w(i) \in [t]$. 
		Then $\displaystyle v= \sum_{j=1}^{\rho} \gamma_j w_j$ for some $\gamma_j \in \Fqm$. For each $\ell \in [t]$, let the set $B_\ell=\{u_{\ell,1},\ldots,u_{\ell,n_\ell}\}$ be an $\Fq$-basis of $\mU_\ell$.  
		For each $i$, there exist $\tau_{\ell,i,r} \in \Fq$ such that
		$\displaystyle w_i = \sum_{\ell=1}^t \sum_{r=1}^{n_\ell} \tau_{\ell,i,r}u_{\ell,r}$, where $\tau_{\ell,i,r}=0$ whenever $\ell \neq\textup{w}_{\textup{srk}}(i)$.
		
		We may now express $v$ as follows:
		\[
		v = \sum_{\ell=1}^t \sum_{r=1}^{n_\ell} u_{\ell,r} \sum_{i=1}^\rho \gamma_i \tau_{\ell,i,r}.
		\]
		Define the matrices $\tau^{(\ell)} = (\tau_{\ell,i,r}) \in \Fq^{\rho \times n_\ell}$, and hence define:
		\begin{align*}
			\lambda =(\lambda_1,\ldots,\lambda_t) &:= (\gamma_1,\ldots,\gamma_\rho) [\tau^{(1)} | \cdots | \tau^{(t)}] \in \F_{q^m}^{1\times n_1}\times \cdots\times \F_{q^m}^{1\times n_t}
		\end{align*}
		Each matrix $\tau^{(\ell)}$ has every $j$-th row all-zeroes if $w(j)\neq \ell$ and, furthermore, there are at most $\rho$ distinct nonzero matrices $\tau^{(\ell)}$.
		We have
		$\lambda_j = \gamma_i \tau^{(j)}_i$, for $i,j$ satisfying $w(i)=j$.
		It follows that $\textup{w}_{\textup{srk}}(\lambda) \leq \rho$.\\
		
		The proofs that (b) implies (c) and (c) implies (a) are very similar to those of \cite[Theorem 2.3]{bonini2022saturating}.\\ 
	\end{proof}	
	\begin{definition}
		Let $\mU$ be an $[\mathbf{n},k]_{q^m/q}$ system. For each positive integer $\rho$, 
		we define 
		\[
		{\mathbb S}_{\rho}(\mU):=\bigcup_{\substack{(\mathcal{S}_i: i \in [t]) :\:\mathcal{S}_i 
				\leq_{\F_q}\mathcal{U}_i,\\\sum_{i=1}^t\dim_{\F_q}\mathcal{S}_i\leq \rho}}
		\left(\bigcup_{i=1}^t (\mathcal{S}_i \otimes {\F_{q^m}})\right).
		\] 
	\end{definition}
	It is immediate from Theorem \ref{th:tfaeabc} that $\mU$ is sum-rank-$\rho$-saturating
	if $\rho$ is the least integer satisfying $\F_{q^m}^k = {\mathbb S}_{\rho}(\mU)$.
	
	The following statement is the sum-rank analogue of \cite[Theorem 2.5]{bonini2022saturating}. The proof is very similar and hence is omitted.
	
	\begin{theorem}
		Let $\mU$ be an $[\mathbf{n},k]_{q^m/q}$ system associated with a code $\C$. The following are equivalent.
		\begin{itemize}
			\item[{\rm (a)}] $\mU$ is sum-rank-$\rho$-saturating.
			\item[{\rm (b)}] $\rho_{srk}(\C^\perp)=\rho$.
		\end{itemize}
	\end{theorem}

	\begin{definition}\label{def:dirsum}
		For $i=1,2$, let $\mU_i$ be a sum-rank-$\rho_i$-saturating $[{\bf n}_i,k_i]_{q^m/q}$ system associated with a code $\C_i$ that has generator matrix $G_i$. We define the \emph{direct sum} of $\mU_1$ and $\mU_2$, which we denote by $\mU_1\oplus \mU_2$, to be the $[({\bf n}_1,{\bf n}_2),k_1+k_2]_{q^m/q}$ system associated with the direct sum of $\C_1$ and $\C_2$, i.e. the code whose generator matrix is 
		\[
		G_1\oplus G_2:=\begin{bmatrix}G_1 & 0\\ 0  &G_2\end{bmatrix}.
		\]
	\end{definition}
	
	It is straightforward to establish the following (c.f. \cite{bonini2022saturating}).
	
	\begin{theorem}
		For $i\in [t]$, let $\mU_i$ be a sum-rank-$\rho_i$-saturating $[{\bf n}_i,k_i]_{q^m/q}$ system.
		Then $\mU_1 \oplus \cdots \oplus \mU_t$ is an $[({\bf n}_1,\dots,{\bf n}_t),k_1+\cdots +k_t]_{q^m/q}$ system and is sum-rank-$\rho$-saturating, for some $\rho \leq \rho_1+\cdots +\rho_t$.
	\end{theorem}

	\begin{definition}
		A sum-rank-$\rho$-saturating system $\mU_1 \oplus \cdots \oplus \mU_t$ is called \emph{reducible} if there exists $i\in\{1,\ldots,t\}$ such that the system $\mU_1 \oplus \cdots \mU_{i-1} \oplus \mU_{i+1} \cdots \oplus \mU_t$ is sum-rank-$\rho$-saturating. Otherwise, the system is called \emph{irreducible}.
	\end{definition}
	
	\section{Bounds on the dimension of sum-rank saturating systems}\label{sec:bounds}
	
	As in the case of Hamming-metric and rank-metric codes, it is interesting to know the shortest length of any sum-rank metric code of a given dimension and covering radius $\rho$, or equivalently, the least rank of any sum-rank-$\rho$-saturating system in a given vector space. 
	
	We start with a bound which follows from the geometric characterisation of our systems. In the proof, we will use the following well-known estimates:
	\begin{align}
		\qbin{a}{b}{q}&<f(q)\, q^{b(a-b)}, && \hspace{-3em} \mbox{for } a,b \in \mathbb N, \label{est1}\\
		q^{e_1}+\ldots+q^{e_r}&< \frac{q}{q-1}q^{e_r}, &&\hspace{-3em} \mbox{for } e_i \in \mathbb Z, \ 0\leq e_1<\ldots<e_r. \label{est3}
	\end{align}
	where $f(q)=\prod_{i=1}^{+\infty}(1-q^{-i})^{-1}$.
	
	For each $n=(n_1,\dots,n_t)\in {\mathcal N}$, we define 
	$$\left\|n\right\|:=\sum_{1\leq i<j\leq t}\left(n_j-n_i\right)^2.$$

	\begin{theorem}\label{thm:bound}
		Let $\mathcal{U}$ be a sum-rank-$\rho$-saturating $[\mathbf{n},k]_{q^m/q}$ system.
		Then
		\[
		q^{m\rho}\sum_{\bs \in {\mathcal N},|s|=\rho} \qbin{\bn}{\bs}{q} \geq q^{mk}.
		\]
		In particular, 
		\begin{equation}\label{eq:lowerbound}
			\frac{1}{4t}\cdot \left\|n\right\|+\frac{\rho(|n|-\rho)}{t}+2t\geq m(k-\rho).    
		\end{equation}
	\end{theorem}
	\begin{proof}
		
		Since $\mU$ is a sum-rank-$\rho$-saturating $[\mathbf{n},k]_{q^m/q}$ system, we have that ${\mathbb S}_\rho(\mU) = \F_{q^m}^k$. Therefore
		\[
		q^{m\rho}\sum_{\bs \in {\mathcal N},|s|=\rho} \qbin{\bn}{\bs}{q} \geq | {\mathbb S}_\rho(\mU)|=  |\F_{q^m}^k|=q^{mk}.
		\]
		By \eqref{est1}, we have:
		\begin{align*}
			\sum_{\bs \in {\mathcal N},|s|=\rho} \qbin{\bn}{\bs}{q}
			< \sum_{\bs \in {\mathcal N},|s|=\rho} \prod_{i=1}^t f(q)q^{s_i(n_i-s_i)}
			= f(q)^t\sum_{\bs \in {\mathcal N},|s|=\rho}  q^{\sum_{i=1}^ts_i(n_i-s_i)}
		\end{align*}
		We proceed by studying the following quantity:
		\[p(s_1,\ldots,s_{t-1}):=\sum_{i=1}^{t-1}s_i(n_i-s_i)+\left(\rho-\sum_{i=1}^{t-1}s_i\right)\left(n_{t}-\rho+\sum_{i=1}^{t-1}s_i\right).\]
		Notice that
		\[\frac{\partial p}{\partial  s_i}(s_1,\ldots,s_{t-1})=n_i-n_t+2\rho-4s_i-2\sum_{j\neq i}s_j.\]
		Since,
		\[(t-1)\frac{\partial p}{\partial  s_i}(s_1,\ldots,s_{t-1})-\sum_{j\neq i}\frac{\partial p}{\partial  s_j}(s_1,\ldots,s_{t-1})=(t-1)n_i-\sum_{j\neq i}n_j+2\rho-2ts_i,\]
		and since we easily see that the maximum of $p(s_1,\ldots,s_{t-1})$ is achieved for
		\[s_i=\frac{1}{2t}\left((t-1)n_i-\sum_{j\neq i}n_j+2\rho\right),\]
		we get:
		\[p(s_1,\ldots,s_{t-1})=\frac{1}{4t}\cdot \sum_{1\leq i<j\leq t}\left(n_j-n_i\right)^2+\frac{\rho(|n|-\rho)}{t}.\]
		Therefore, by \eqref{est3}, we get
		\[
		\frac{q\cdot f(q)^t}{q-1}\cdot
		q^{ \left(\frac{1}{4t}\cdot \left\|n\right\|+\frac{\rho(|n|-\rho)}{t} \right)}> q^{m(k-\rho)}.
		\]
		When $q> 2$, since
		$\displaystyle \frac{q\cdot f(q)^t}{q-1}\leq q^{t}$,
		we get
		\[\frac{1}{4t}\cdot \left\|n\right\|+\frac{\rho(|n|-\rho)}{t}+t\geq m(k-\rho),\]
		while if $q=2$ we obtain 
		\[\frac{1}{4t}\cdot \left\|n\right\|+\frac{\rho(|n|-\rho)}{t}+2t\geq m(k-\rho).\]
	\end{proof}

	\begin{remark}
		We have that $f(q)\longrightarrow 1$ as $q \longrightarrow \infty$, and so asymptotically $\frac{q f(q)^t}{q-1} \longrightarrow 1$ as $q \longrightarrow \infty$. For this reason, as $q$ grows, we may replace (\ref{eq:lowerbound}) with
		\begin{equation*}
			\frac{1}{4t}\cdot \left\|n\right\|+\frac{\rho(|n|-\rho)}{t}\geq m(k-\rho), 
		\end{equation*}
		for sufficiently large $q$.
		Indeed, even for relatively small values of $q$, $\frac{q f(q)^t}{q-1}$ takes values much smaller than $q$, for $t$ not exceeding $q$. For example, for $q=211,t=20$, we have $\frac{q f(q)^t}{q-1} \approx 1.105407$; for $q=111,t=111$ we have $\frac{q f(q)^t}{q-1} \approx 2.780617$.  
	\end{remark}

	\begin{remark}
		For $t=1$ (the rank-metric case), the bound coincides asymptotically with the one obtained in \cite{bonini2022saturating} (while for small $q$, in \cite{bonini2022saturating} the rough estimate could be avoided). For $n_1=\cdots=n_t=n$, we have:
		\begin{equation}\label{eq:asyeq}
			N=tn\geq \frac{tm}{\rho}(k-\rho)+\rho - \frac{2t^2}{\rho}.
		\end{equation}
	\end{remark}

	\begin{lemma}
		Let $I \in \{1,\dots,t\}$. Let $n=(n_1,\dots,n_t),s=(s_1,\dots,s_t)=(0,\dots,0,s_I,s_{I+1}\dots,s_t) \in {\mathcal N}$ satisfy the following properties:
		\begin{enumerate}
			\item $n_1 \geq n_2 \geq \cdots \geq n_t$,
			\item $\displaystyle \sum_{j=1}^t s_j = 0$,
			\item $s_I \geq 0$ and $s_j \leq 0, j \in \{I+1,\dots,t\}$.
		\end{enumerate}
		Then $\left\| n \right \| \leq \left\| n +s\right \|$.
	\end{lemma}
	
	\begin{proof}
		We have that
		\begin{eqnarray*}
			\left\| n +s\right \| &=& \left\| n \right \| + \left\| s\right \| 
			+ 2\sum_{i=1}^t\sum_{j=i+1}^t (n_i-n_j)(s_i-s_j). 
		\end{eqnarray*}
		We will show that $\displaystyle S=\sum_{i=1}^t\sum_{j=i+1}^t (n_i-n_j)(s_i-s_j) \geq 0$.
		We have:
		\begin{eqnarray*}
			S & = & \sum_{i=1}^{I-1}\sum_{j=i+1}^t (n_i-n_j)(-s_j) 
			+ \sum_{j=I+1}^t (n_I-n_j)(s_I-s_j) + \sum_{i=I+1}^t\sum_{j=i+1}^t (n_i-n_j)(s_i-s_j).
		\end{eqnarray*}
		The first term on the right-hand-side of the above equation is clearly non-negative, so we consider:
		\begin{eqnarray*} 
			T&=&\sum_{j=I+1}^t (n_I-n_j)(s_I-s_j) + \sum_{i=I+1}^t\sum_{j=i+1}^t (n_i-n_j)(s_i-s_j) 
			\\
			&=&s_I \sum_{j=I+1}^t(n_I-n_j)- \sum_{i=I+1}^t s_i(n_I-n_i)
			+  \sum_{i=I+1}s_i \sum_{j=i+1}^t (n_i-n_j) -\sum_{i=I+1}^t \sum_{j=i+1}^t s_j (n_i-n_j)\\
			&=&s_I \sum_{j=I+1}^t(n_I-n_j)- \sum_{i=I+1}^t s_i(n_I-n_i
			+  \sum_{j=i+1}^t (n_i-n_j)) -\sum_{i=I+1}^t \sum_{j=i+1}^t s_j (n_i-n_j).
		\end{eqnarray*}
		Since $s_I \geq 0, s_j \leq 0$ for $j \geq I+1$, and since the $n_i$ form a non-increasing sequence, we deduce that $T \geq 0$. It follows that $\left\| n +s\right \| - \left\| n\right \| \geq 0$.
	\end{proof}
	
	This means that, for fixed $\rho, t, N$, the left-hand-side of \eqref{eq:lowerbound} takes its minimum and maximum values for $n_1=\cdots=n_t$ (when $t$ divides $N$) and for $n_1=N-t+1,n_2=\cdots=n_t=1$, respectively. That is to say, if $n=(N-t+1,1,\dots,1)$, then for any $n'=(n'_1,\dots,n'_t) \in {\mathcal N}$, we have 
	$n-n' = (N-t+1-n'_1,1-n'_2,\dots,1-n'_t)$. If the $n'_i$ form a non-increasing sequence whose sum is $N$ then we have that $\left\| n\right \| \geq  \left\| n'\right \|$, by the previous lemma. This gives meaning to the following definition.
	
	\begin{definition}
		Let $t$ be a positive integer. We define the \emph{shortest length} of a sum-rank-$\rho$-saturating system $\mU=(\mU_1,\ldots,\mU_t)$ in $\F_{q^m}^k$ to be
		\[
		s_{q^m/q}(k,\rho,t):=\min\left\{\sum_{i=1}^t \dim(\mU_i) : \mU_i \leq_{\Fq} \Fqm^k,(\mU_1,\dots,\mU_t) \text{ is sum-rank-}\rho\text{-saturating} \right\},
		\]
		i.e. it is the minimal sum of the $\F_q$-dimensions of the $\mU_i$, $i\in \{1,\ldots,t\}$.\\ 
		We define the \emph{homogeneous shortest length} of a sum-rank-$\rho$-saturating system $\mU=(\mU_1,\ldots,\mU_t)$ in $\F_{q^m}^k$ to be:
		\[
		s_{q^m/q}^{\rm hom}(k,\rho,t):=\min\left\{tn : \mU_i \leq_{\Fq} \Fqm^k,\dim(\mU_i)=n,(\mU_1,\dots,\mU_t) \text{ is sum-rank-}\rho \text{-saturating} \right\},
		\]
		i.e. it is the minimal sum of the $\F_q$-dimensions of the $\mU_i$, $i\in \{1,\ldots,t\}$, with the additional hypothesis that each $\mU_i$ has the same dimension $n$. 
	\end{definition}
	
	Let $\mathcal{U}=(\mathcal{U}_1,\ldots,\mathcal{U}_t)$ be a sum-rank saturating system with generator matrix $G = [G_1 | \cdots | G_t]$. Consider the system  $\mathcal{U}^\prime=(\mathcal{U}_1,\ldots, \mathcal{U}_{t-2},\mathcal{U}_{t-1}^\prime)$, which has generator matrix $G = [G_1 | \cdots | G_{t-2} | G_{t-1}^\prime]$, where $G_{t-1}^\prime$ is a matrix whose columns are a union of $\F_q$-bases of $ \mathcal{U}_{t-1}$ and $\mathcal{U}_t$. Since $ \mathcal{U}_{t-1}+ \mathcal{U}_t  = \mathcal{U}_{t-1}^\prime $ we have $\dim_{\F_q}(\mathcal{U}_{t-1}^\prime)\le \dim_{\F_q}(\mathcal{U}_{t-1})+\dim_{\F_q}(\mathcal{U}_{t})$, while $\rho(\mathcal{U}^\prime)\le \rho(\mathcal{U})$. 
	
	For this reason we have the following proposition.
	
	\begin{proposition}[Monotonicity in $t$]
		We have that $s_{q^m/q}(k,\rho,t)\le s_{q^m/q}(k,\rho,t+1)$.
	\end{proposition}
	
	\begin{lemma}\label{lem:notscattered}
		Let $\mathcal{U}=(\mU_1,\dots,\mU_t)$ be a sum-rank-$\rho$-saturating $[\mathbf{n},k]_{q^m/q}$ system. 
		Suppose for some $i \in [t]$, $L_{\mU_i}$ is not scattered.
		Let $\mU'_{i} = \langle u_{i,1},\dots,u_{i,{n_i-1}} \rangle_{\F_q}$ for some 
		$\F_q$-basis $\{u_{i,1},\dots,u_{i,n_i}\}$ of $\mU_i$ such that $u_{i,n_i} \in  \langle \lambda u_{i,1},\dots,\lambda u_{i,{n_i-1}}\rangle_{\F_q}$ for some $\lambda \in \F_{q^m}$. 
		Then $\mU'=(\mU_1,\dots,\mU'_{i},\dots,\mU_t)$ is a sum-rank-$\rho'$-saturating $[{\bf n}',k]_{q^m/q}$ system satisfying $\rho' \leq \rho+1$ and 
		${\bf n}' = (n_1,\dots,n_{i}-1,\dots,n_t)$.
	\end{lemma}
	
	\begin{proof}
		The statement follows as a direct consequence of \cite[Lemma 4.5]{bonini2022saturating}, which gives that if $L_{\mU_i}$ is scattered then $\mU_i'$ is an $[n_i-1,k_i]_{q^m/q}$ rank-$\rho_i'$-saturating system satisfying $\rho_i' \leq \rho_i+1$.
	\end{proof}
	
	More generally, we have the following.
	
	\begin{lemma}\label{lem:notsumscattered}
		Let $\mathcal{U}=(\mU_1,\dots,\mU_t)$ be a sum-rank-$\rho$-saturating $[\mathbf{n},k]_{q^m/q}$ system. 
		Suppose that for each $i \in [t]$, $\mU_i$ has an $\Fq$-basis $\{u^{(i)}_1,\dots,u^{(i)}_{n_i}\}$ 
		such that
		\[
		u^{(t)}_{n_t} = \lambda \sum_{i\in S} \sum_{\substack{j=1,\\j \neq n_t}}^{n_i} a^{(i)}_j u^{(i)}_j,
		\]
		for some $\lambda \in \F_{q^m}$, $a^{(i)}_j \in \Fq$ and $S \subseteq [t]$.
		Then $\mU'=(\mU_1,\dots,\mU_{t-1},\mU'_t)$ is a sum-rank-$\rho'$-saturating $[{\bf n}',k]_{q^m/q}$ system satisfying $\rho' \leq \rho+|S|$ and ${\bf n}' = (n_1,\dots,n_{t-1},n_t-1)$.
	\end{lemma}
	\begin{proof}
		Let $v \in \Fqm^k$. As in the proof of Theorem \ref{th:tfaeabc}, there exist $\gamma_1,\dots \gamma_\rho \in \Fqm$ such that
		\[
		v = \sum_{\ell=1}^t \sum_{r=1}^{n_\ell} u_r^{(\ell)} \sum_{i=1}^\rho \gamma_i \tau^{(\ell)}_{i,r}.
		\]
		for some $\tau^{(\ell)}_{i,r} \in \Fq$ with $\tau^{(\ell)}_{i,r}$ nonzero for at most one pair $(\ell,i)$ and at most $\rho$ of the matrices $\tau^{(\ell)}:=(\tau^{(\ell)}_{i,r}) \in \Fq^{\rho \times n_\ell}$ are nonzero.  
		Define $a^{(\ell)}_r:=0$ for all $r\in [n_\ell]$ whenever $\ell \notin S$.
		This yields that
		\begin{eqnarray*}
			v & = &\sum_{\ell=1}^t \sum_{r=1}^{n_\ell} u_r^{(\ell)} \sum_{i=1}^\rho \gamma_i \tau^{(\ell)}_{i,r} 
			\\
			& = & \sum_{\ell=1}^{t-1} \sum_{r=1}^{n_\ell} u_r^{(\ell)} \sum_{i=1}^\rho \gamma_i \tau^{(\ell)}_{i,r}
			+ \sum_{r=1}^{n_t-1} u_r^{(t)} \sum_{i=1}^\rho \gamma_i \tau^{(t)}_{i,r} + 
			\lambda \sum_{\ell=1}^t \sum_{\substack{r=1,\\r \neq n_t}}^{n_\ell} a^{(\ell)}_r u^{(\ell)}_r \sum_{i=1}^\rho \gamma_i \tau^{(t)}_{i,n_t} \\
			& = & \sum_{\ell=1}^{t-1} \sum_{r=1}^{n_\ell} u_r^{(\ell)}\left(\sum_{i=1}^\rho \gamma_i \tau^{(\ell)}_{i,r} +  a^{(\ell)}_r \lambda \sum_{i=1}^\rho \gamma_i \tau^{(t)}_{i,n_t}\right) + 
			\sum_{r=1}^{n_t-1} u_r^{(t)}\left( \sum_{i=1}^\rho \gamma_i \tau^{(t)}_{i,r} +   a^{(t)}_r \lambda \sum_{i=1}^\rho \gamma_i \tau^{(t)}_{i,n_t} \right)
			\\
			& = & \sum_{\ell=1}^{t-1} \sum_{r=1}^{n_\ell} u_r^{(\ell)}\left(\sum_{i=1}^\rho \gamma_i \tau^{(\ell)}_{i,r} +  a^{(\ell)}_r \gamma_{\rho+1}\right) + 
			\sum_{r=1}^{n_t-1} u_r^{(t)}\left( \sum_{i=1}^\rho \gamma_i \tau^{(t)}_{i,r} +   a^{(t)}_r \gamma_{\rho+1} \right)
			\\
			& = & \sum_{\ell=1}^{t-1} \sum_{r=1}^{n_\ell} u_r^{(\ell)}\sum_{i=1}^{\rho+1} \gamma_i \tau^{(\ell)}_{i,r} + 
			\sum_{r=1}^{n_t-1} u_r^{(t)}\sum_{i=1}^{\rho+1} \gamma_i \tau^{(t)}_{i,r},
			\\
			& = &\sum_{\ell=1}^t \sum_{r=1}^{n'_\ell} u_r^{(\ell)} \sum_{i=1}^{\rho+1} \gamma_i \tau^{(\ell)}_{i,r}
		\end{eqnarray*}
		where $n'_t=n_t-1$, $n'_\ell=n_\ell$ if $\ell < t$;
		$\displaystyle \gamma_{\rho+1}=\lambda \sum_{i=1}^\rho \gamma_i \tau^{(t)}_{i,n_t}$, and $\tau^{(\ell)}_{\rho+1,r}=a^{(\ell)}_r$ for each $\ell \in [t]$.
		Then $$\displaystyle v = \sum_{\ell=1}^t G^{(\ell)} \lambda_\ell^T = G \lambda^T,$$ where
		$$G=[G^{(1)}|\cdots| G^{(t)}], G^{(\ell)} = [u^{(\ell)}_1,\dots,u^{(\ell)}_{n'_\ell}],$$ 
		$$(\lambda_1,\dots,\lambda_t) = (\gamma_1,\dots,\gamma_{\rho+1})[\hat{\tau}^{(1)} | \cdots|\hat{\tau}^{(t)}]= [ \gamma\hat{\tau}^{(1)} | \cdots| \gamma\hat{\tau}^{(t)}],$$ and $\hat{\tau}^{(\ell)}=(\tau^{(\ell)}_{i,r}) \in F_q^{(\rho+1) \times n_\ell} $ for each $\ell \in [t]$.
		Now consider the value of 
		$W(\lambda) = \sum_{\ell=1}^t \rk(\lambda_\ell) =  \sum_{\ell=1}^t \mathrm{w}_\rk(\gamma \hat{\tau}^{(\ell)})$.
		Let $\ell \in S$. If $\tau^{(\ell)}$ is not the all-zero matrix, then $\hat{\tau}^{(\ell)}$ has at most $2$ nonzero rows and so $$\mathrm{w}_\rk(\lambda_\ell)=\mathrm{w}_\rk(\gamma \hat{\tau}^{(\ell)}) \leq 2.$$ Otherwise $\lambda_\ell = \gamma_{\rho+1} (\tau^{(\ell)}_{\rho+1,1},\dots,\tau^{(\ell)}_{\rho+1,n_\ell})$ has rank weight at most 1.
		If $\ell \notin S$ then $\mathrm{w}_\rk(\lambda_\ell)\leq 1$.
		It follows that $w(\lambda) \leq \rho + |S|$.
	\end{proof}
	
	In particular, Lemma \ref{lem:notscattered} follows as a special case of Lemma \ref{lem:notsumscattered}.
	
	We have the following observations on the monotonicity of $s_{q^m/q}(k,\rho,t)$. The proofs are similar to those of \cite[Theorem 4.6]{bonini2022saturating}.
	
	\begin{theorem}[Monotonicity in $\rho$]\label{th:monrho}
		Let $|\mathbf{n}|>k$. The following hold.
		\begin{enumerate}
			\item $s_{q^m/q}(k,\rho,t)\le s_{q^m/q}(k,\rho+1,t)$.
			\item $s_{q^m/q}(k,\rho,t)\leq s_{q^m/q}(k+1,\rho,t)-1$.
			\item $s_{q^m/q}(k+1,\rho+1,t)\leq s_{q^m/q}(k,\rho+1,t)+1$.
		\end{enumerate}
	\end{theorem}
	
	\begin{proof}
		Let $\mU=(\mU_1,\dots,\mU_t)$ be an $[\mathbf{n},k]_{q^m/q}$ sum-rank-$\rho$-saturating system such that $|\mathbf{n}|=s_{q^m/q}(k,\rho,t)$. 
		Without loss of generality, we may assume that a generator matrix of a code associated with $\mU$ has the form:
		\begin{equation}\label{eq:sfm}
			G=[G_1,\dots,G_t]=
			\left[
			\begin{array}{cc|cc|cc|cc}
				I_{k_1}& A_{1,1}& 0      & A_{1,2} & \cdots & \cdots &0      & A_{1,t}  \\
				0      &   0    & I_{k_2}& A_{2,2} & \cdots & \cdots &0      &A_{2,t}  \\
				0      &   0    & 0      & 0       & \cdots & \cdots &0      &A_{3,t}  \\
				\vdots &\vdots  & \vdots &\vdots   & \vdots & \vdots &\vdots &\vdots\\
				0      &   0    & 0      & 0       & \cdots & \cdots &I_{k_t}& A_{t,t}  
			\end{array}
			\right]
		\end{equation}
		We may assume that $n_1 > k_1 \geq 1$, since otherwise we can permute the submatrices $G_i$ and apply arbitrary elementary row operations to the block-permuted matrix and perform $\Fq$-linear column operations within each $G_i$ to find a matrix of the required form.
		Over all such choices of $\mU$ and $G$, let $G$ be one such that the rightmost column of $A_{1,1}$, is a vector $y \in \Fqm^{k_1}$ that has minimal rank weight.
		If $\mathrm{w}_{\rk}(y)=1$, then by Lemma \ref{lem:notscattered}, there exists an
		$[\mathbf{n'},k]_{q^m/q}$ sum-rank-$\rho'$-saturating system with $|\mathbf{n'}|=|\mathbf{n}|-1$ and $\rho'\leq\rho+1$; indeed we then have $\rho'=\rho+1$ by the minimality of $|\mathbf{n}|$. 
		In this case we have $s_{q^m/q}(k,\rho+1,t) \leq s_{q^m/q}(k,\rho,t)-1$.
		
		If $\mathrm{w}_{\rk}(y)=\ell\geq 2$, 
		then $y=\sum_{i=1}^\ell\alpha_i y^{(i)}$ for an $\Fq$-basis $\{\alpha_1,\dots,\alpha_\ell\} \subseteq \Fqm$ of the $\Fq$-span of the coefficients of $y$ and $y^{(i)} \in \Fq^{k_1}$. 
		A straightforward computation verifies that the matrix $G^{(1)}$ found by replacing $y$ in $A_{1,1}$ with $x^{(1)}:=\sum_{i=1}^{\ell-1}\alpha_i y^{(i)}$
		yields an $[\mathbf{n},k]_{q^m/q}$ sum-rank-$\rho^{(1)}$-saturating system with $\rho^{(1)}\leq\rho+1$. If $\rho^{(1)}=\rho+1$ then the statement of the theorem holds, so suppose otherwise. If $\rho^{(1)}=\rho$ then we arrive at a contradiction by the minimality of $\mathrm{w}_{\rk}(y) > \mathrm{w}_{\rk}(x^{(1)})$, so suppose $\rho^{(1)} \leq \rho -1$. 
		The matrix $G^{(2)}$ found by replacing $y$ with $x^{(2)} :=\sum_{i=1}^{\ell-2}\alpha_i y^{(i)}$ yields an $[\mathbf{n},k]_{q^m/q}$ sum-rank-$\rho^{(2)}$-saturating system with $\rho^{(2)}\leq\rho^{(1)}+1 \leq \rho$. As before, by the minimality of $\mathrm{w}_{\rk}(y) > \mathrm{w}_{\rk}(x^{(2)})$, we have $\rho^{(2)} \neq \rho$ and so $\rho^{(2)}\leq \rho-1$.
		Repeated applications of the above argument lead to a sequence of 
		$[\mathbf{n},k]_{q^m/q}$ sum-rank-$\rho^{(i)}$-saturating systems with $\rho^{(i)}\leq \rho-1$ for each $i \in [\ell-1]$. The final matrix $G^{(\ell-1)}$
		in this sequence has rightmost column of $A_{1,1}$ equal to $x^{(\ell-1)} = \alpha_1 y^{(1)}$, so we may apply Lemma \ref{lem:notscattered} to see that 
		deleting this column from $G^{(\ell-1)}$ results in an $[\mathbf{n'},k]_{q^m/q}$ sum-rank-$\rho$-saturating system with $|\mathbf{n'}|=|\mathbf{n}|-1$, giving a contradiction. 
		We deduce that $\rho^{(1)}=\rho+1$, in which we have $s_{q^m/q}(k,\rho+1,t) \leq s_{q^m/q}(k,\rho,t)$.  
		This proves 1.

		The proofs that 2 and 3 hold are very similar to the rank-metric case and are omitted (see \cite[Theorem 4.6]{bonini2022saturating}).
	\end{proof}

	\begin{definition}
		For each $i\in \{1,2\}$, let $\mU^{(i)}$ be an 
		$[{\bf n^{(i)}},k_i]_{q^m/q}$ system, associated with an 
		$[{\bf n^{(i)}},k_i]_{q^m/q}$ sum-rank metric code $\mC_i$.
		Let $f:\F_{q^m}^{{\bf n^{(1)}}} \longrightarrow \F_{q^m}^{{\bf n^{(2)}}}$ be an $\F_{q^m}$-linear map.
		The code 
		\[\mC:=\{(u,f(u)+v) : u \in \C_1, v \in \mC_2\}\] 
		is an $[({{\bf n^{(1)}}},{{\bf n^{(2)}}}),k_1+k_2]_{q^m/q}$ code, which we call the $f$-sum of $\mC_1$ and $\mC_2$. The
		$[({{\bf n^{(1)}}},{{\bf n^{(2)}}}),k_1+k_2]_{q^m/q}$ system associated with $\mC$ is called the $f$-sum of $\mU^{(1)}$ and $\mU^{(2)}$, which we denote by 
		$\mU^{(1)} \oplus_{f} \mU^{(2)}$.    
		If $f$ is the zero map, we write $\mU^{(1)} \oplus \mU^{(2)}$, and call it the direct sum; if $f$ is the identity map, we write $\mU^{(1)} \oplus_\iota \mU^{(2)}$ and call it the Plotkin-sum of $\mU^{(1)}$ and $\mU^{(2)}$.
	\end{definition}
	
	\begin{proposition}
		\label{prop:DirectSum}
		For each $i\in \{1,2\}$, let ${\bf n}^{(i)} = ({\bf n}^{(i)}_1,\dots,{\bf n}^{(i)}_{t_i})$, and let $\mU^{(i)}$ be an 
		$[{\bf n^{(i)}},k_i]_{q^m/q}$ 
		sum-rank-$\rho_i$-saturating system, associated with an $[{\bf n^{(i)}},k_i]_{q^m/q}$ code $\mC_i$.
		Let $f:\F_{q^m}^{{\bf n^{(1)}}} \longrightarrow \F_{q^m}^{{\bf n^{(2)}}}$ be an $\F_{q^m}$-linear map.
		Then $\mU^{(1)} \oplus_{f} \mU^{(2)}$ is an 
		$[({{\bf n^{(1)}}},{{\bf n^{(2)}}}),k_1+k_2]_{q^m/q}$ system 
		that is sum-rank-$\rho$-saturating, where 
		$\rho\leq \rho_1+\rho_2.$
		In particular, if $\rho_1+\rho_2\leq \min\{k_1+k_2,m\}$, then
		\[s_{q^m/q}(k_1+k_2,\rho_1+\rho_2,t_1+t_2) \leq s_{q^m/q}(k_1,\rho_1,t_1)+s_{q^m/q}(k_2,\rho_2,t_2).\]
	\end{proposition}	
	
	\begin{proof}
		Let $\mU'$ be the $[{\bf n}^{(2)},k_1]_{q^m/q}$ system associated with $f(\mU)$. Then 
		$\mU^{(1)}+\mU'$ is a sum-rank-$\rho'$-saturating $[({\bf n}^{(1)},{\bf n}^{(2)}),k_1]_{q^m/q}$ system, satisfying $\rho' \leq \rho_1$. Therefore,
		\[
		\F_{q^m}^{k_1+k_2} = {\mathbb S}_{\rho'}((\mU^{(1)}+\mU')\oplus {\bf 0}_{k_2}) \cup 
		{\mathbb S}_{\rho_2}( {\bf 0}_{k_1}\oplus \mU_2) = {\mathbb S}_{\rho'+\rho_2}(\mU^{(1)} \oplus_f \mU_2),
		\]
		and so $\mU^{(1)} \oplus_f \mU_2$ is an $[({\bf n}^{(1)},{\bf n}^{(2)}),k_1+k_2]_{q^m/q}$ system that is sum-rank-$\rho$-saturating for $\rho \leq \rho_1+\rho_2$.
		The rest now follows by choosing each $\mU^{(i)}$ to have length $s_{q^m/q}(k_i,\rho_i,t_i)$ and applying Theorem \ref{th:monrho}. 
	\end{proof}
	
	\begin{theorem}
		Let $\F_{q^m}=\F_q[\alpha]$, $r\geq 1$, $h\geq r$ and
		\[A_{h,r}:=\left[ \begin{array}{c|c|c|c|c}
			I_{r} & \mathbf{0}   & \mathbf{0}     & \cdots     & \mathbf{0}\\
			\hline
			\mathbf{0} & I_{h-r} & \alpha I_{h-r} & \cdots & \alpha^{m-1}I_{h-r} \end{array}\right]\]
		Then
		\[G_t:=\underbrace{\left[ \begin{array}{c|c|c|c}
				A_{h,r} & \mathbf{0}   &  \cdots     & \mathbf{0}\\
				\hline
				\mathbf{0} & A_{h,r} &  \cdots     &  \mathbf{0}\\
				\hline
				\vdots & \vdots &  \ddots     &  \vdots \\
				\hline
				\mathbf{0}  & \mathbf{0}   &  \cdots     & A_{h,r}
			\end{array}\right]}_{t\text{ times}}\]
		generates a homogeneous sum-rank-$rt$-saturating system. So
		\[s_{q^m/q}^{\rm hom}(th,tr,t)\leq t(m(h-r)+r).\]
	\end{theorem}
	
	\begin{proof}
		From \cite[Theorem 4.4]{bonini2022saturating}, we have that 
		$A_{h,r}$ is the generator matrix of a code associated with an $r$-rank saturating $[m(h-r)+r,h]_{q^m/q}$ system $\mU_{h,r}$.
		The matrix $G_t$ is the generator matrix of a code
		associated with the direct sum of $t$ copies of $\mU_{h,r}$, which from Proposition \ref{prop:DirectSum} is a sum-rank-$\rho$-saturating
		$[t(m(h-r)+r),th]_{q^m/q}$ system, with $\rho \leq tr$.
		It is not hard to see that $\rho=tr$. Let $v = (v^{(1)},\dots,v^{(t)}) \in \F_{q^m}^{th}$ such that each $v^{(i)} \in \F_{q^m}^h$ has its first $r$ coefficients nonzero. Then any expression of $v$ as an $\F_{q^m}$-linear combination of the columns of $G_t$ requires the use of all its $tr$ columns. 
	\end{proof}
	
	\begin{remark}
		Since \[t\left(\frac{m}{r}(h-r)+r\right)\leq s_{q^m/q}^{\rm hom}(th,tr,t)\leq t(m(h-r)+r),\]
		we see immediately that when $r=1$ the lower and the upper bounds coincide, so that 
		\[s_{q^m/q}^{\rm hom}(th,t,t)=t(m(h-1)+1).\]
	\end{remark}

	\section{Constructions of sum-rank saturating systems}\label{sec:constructions}
	
	In this final section we present some constructions of sum-rank-$\rho$-saturating systems of small $\F_q$-dimension.
	
	\subsection{Sum-rank saturating systems from partitions of the projective space}
	
	We construct sum-rank saturating systems from partitions of the projective space. First observe that if $\mU=(\mU_1,\ldots,\mU_t)$ is such that $L_{\mU_1}\cup\cdots\cup L_{\mU_t}=\PG(k-1,q^m)$, than $\mU$ is sum-rank-$1$-saturating.

	\begin{example}
		In \cite[Theorem 4.28]{hirschfeld1998projective} we get that, if $(m,k)=1$, there exists a partition of $\PG(k-1,q^m)$ into
		\[t=\frac{(q^{mk}-1)(q-1)}{(q^m-1)(q^k-1)}\]
		subgeometries $\PG(k-1,q)$. This gives us a sum-rank $1$-saturating system of total length
		\[k\cdot \frac{(q^{mk}-1)(q-1)}{(q^m-1)(q^k-1)}.
		\]
	\end{example}
	
	We mention another construction of a sum-rank-$\rho$-saturating system based on a partition of $\PG(k,q^m)$ into subgeometries.
	
	\begin{proposition}\label{prop:partition}
		Let $\mathcal{P}=\{\mathcal{P}_i\}_{i\in \{1,\ldots,t\}}$ a partition of $\PG(k-1,q^m)$ into subspaces. Let $k_i$ be a positive integer such that $\mathcal{P}_i\simeq \PG(k_i-1,q^m)$. If $\mU$ is such that each $\mU_i$ is rank-$\rho$-saturating in $\mathcal{P}_i$, then $\mU$ is sum-rank-$\rho'$-saturating with $\rho'\leq \rho$.
	\end{proposition}
	
	\begin{proof}
		Let $P \in \PG(k-1,q^m)$. Then $P \in \mathcal{P}_i$ for some $i$ and $\mU_i$ is a rank-$\rho$-saturating system. In particular, $P$ is in the span of at most $\rho$ elements of $L_{\mU_i}$.
		Therefore, $\mU$ is sum-rank-$\rho'$-saturating with $\rho'\leq \rho$.
	\end{proof}
	
	A partition of the vector space $\Fqm^k$ yields a partition of $\PG(k-1,q^m)$ into subspaces. In \cite{bu1980partitions}, some necessary conditions and constructions of partitions are presented. Thanks to Proposition \ref{prop:partition}, every such partition may be combined with other constructions.

	\subsection{Sum-rank-$(k-1)$-saturating systems from cutting designs}
	
	In this section, we introduce the concept of sum-rank metric minimal codes and we examine their parameters. The geometry of minimal codes has been significant in constructing and establishing bounds in both the Hamming and rank metric, through the so-called \emph{strong blocking sets}. These sets, first introduced in \cite{davydov2011linear} in order to get small saturating sets, are collections of points in the projective space such that the intersection with every hyperplane spans the hyperplane. In \cite{fancsali2014lines}, strong blocking sets are referred to as generator sets and are formed by unions of disjoint lines. They have recently garnered renewed interest in coding theory, especially since \cite{bonini2021minimal}, where they are called \emph{cutting blocking sets} and are utilized to construct minimal codes. Quite surprisingly, they have been demonstrated to be the geometric counterparts of minimal codes \cite{alfarano2019geometric,tang2021full}.
	
	Minimal codes and their geometric counterparts may be introduced also in the context of the sum-rank metric.

	\begin{definition}
		Let $\C$ be an $[\mathbf{n},k]_{q^m/q}$ sum-rank metric code. A nonzero codeword $c \in \C$ is called \emph{minimal} if for every $c'\in \C$ such that $\supp_{\mathbf{n}}(c')\subseteq \supp_{\mathbf{n}}(c)$ then $c'=\lambda c$ for some $\lambda \in \F_{q^m}$. We say that $\C$ is \emph{minimal} if all of its nonzero codewords are minimal. 
	\end{definition}
	
	\begin{definition}
		A system $\mU=(\mU_1,\ldots,\mU_t)\subset \F_{q^m}^k$ is called \emph{cutting} if $L_{\mU_1}\cup \ldots \cup L_{\mU_t}$ is a \emph{strong blocking set} in $\PG(k-1,q^m)$, that is, if
		\[\langle(L_{\mU_1}\cup \ldots \cup L_{\mU_t})\cap \mH\rangle_{\F_{q^m}} = \mH,\]
		for every hyperplane $\mH$ in $\PG(k-1,q^m)$.
	\end{definition}
	
	The following is a generalization of the geometric characterization of minimal codes in the Hamming and in the rank metric.
	
	\begin{theorem}[\!\!{\cite[Corollary 10.25]{santonastaso2022subspace}}]\label{th:correspondence} 
		A sum-rank metric code is minimal if and only if an associated system is cutting.
	\end{theorem}
	
	As in the other metrics, also in the sum-rank one cutting systems give rise to saturating systems.
	
	\begin{theorem}\label{thm:cutting}
		If $\mU$ is a cutting system in $\F_{q^m}^k$, then $\mU$ is a sum-rank-$(k-1)$-saturating system in $\F_{q^{m(k-1)}}^k$.
	\end{theorem}
	
	\begin{proof}
		The system $\mU=(\mU_1,\ldots,\mU_t)$ is cutting in $\F_{q^m}^k$, so that the associated code $\C$ is minimal (by Theorem \ref{th:correspondence}). Then the associated Hamming-metric code to $\C$ defined in \cite[Definition 9.26.]{santonastaso2022subspace} is minimal, by \cite[Corollary 9.27.]{santonastaso2022subspace}. 
		Hence $L_{\mU_1}\cup \ldots\cup L_{\mU_t}$ is a strong blocking set in ${\rm PG}(k-1,q^m)$. Then $L_{\mU_1}\cup \ldots\cup L_{\mU_t}$ is a $(k-2)$-saturating set in ${\rm PG}(k-1,q^{m(k-1)})$ by \cite{davydov2011linear}. By definition, this means that $\mU$ is a sum-rank $(k-1)$-saturating system in $\F_{q^{m(k-1)}}^k$.
	\end{proof}

	\begin{example}
		In \cite[Section 4]{borello2023geometric}, the authors provide bounds on the parameters and constructions of  minimal sum-rank codes. These last have either one or two nonzero weights. Thanks to Theorem \ref{thm:cutting}, these constructions provide more examples of saturating systems in the sum-rank metric.
		
		The doubly extended linearized Reed-Solomon code (see \cite{neri2023geometry}) with parameters $$[((\underbrace{m,\ldots,m}_{q-1\text{ times}},1,1),2]_{q^m/q}$$
		and their geometric dual with parameters $$[(\underbrace{m,\ldots,m}_{q-1\text{ times}},2m-1,2m-1),2]_{q^m/q}$$ are both minimal sum-rank codes (see \cite[Remark 4.6.]{borello2023geometric}). The first ones have length $(q-1)m+2$, which meets the lower bound for the length of minimal sum-rank codes, for any $m$ and the second ones have length $(q-1)m+4m-2$, which is minimal for $m=2$. Note the parameters of these codes do not meet our lower bound.
	\end{example}

	\noindent \textbf{Acknowledgements.}
	We thank the anonymous reviewers for their careful reading of this manuscript.
	The results of this paper are the result of a collaboration that arose within the IRC-PHC Ulysses
	project “Geometric Constructions of Codes for Secret Sharing Schemes”. The second author is partially supported by the ANR-21-CE39-0009 - BARRACUDA (French \emph{Agence Nationale de la Recherche}).
	
	\bibliographystyle{abbrv}
	\bibliography{references.bib}

\begin{thebibliography}{10}

\bibitem{abiad+}
A.~Abiad, A.~P. Khramova, and A.~Ravagnani.
\newblock Eigenvalue bounds for sum-rank-metric codes.
\newblock {\em IEEE Transactions on Information Theory}, 70(7):4843--4855,
  2024.

\bibitem{alfarano2019geometric}
G.~N. Alfarano, M.~Borello, and A.~Neri.
\newblock A geometric characterization of minimal codes and their asymptotic
  performance.
\newblock {\em Advances in Mathematics of Communications}, 16(1):115--133,
  2022.

\bibitem{alfarano2021linear}
G.~N. Alfarano, M.~Borello, A.~Neri, and A.~Ravagnani.
\newblock Linear cutting blocking sets and minimal codes in the rank metric.
\newblock {\em Journal of Combinatorial Theory, Series A}, 192:105658, 2022.

\bibitem{bartoli2023saturating}
D.~Bartoli, M.~Borello, and G.~Marino.
\newblock Saturating linear sets of minimal rank.
\newblock {\em Finite Fields and Their Applications}, 95:102390, 2024.

\bibitem{bonini2021minimal}
M.~Bonini and M.~Borello.
\newblock Minimal linear codes arising from blocking sets.
\newblock {\em Journal of Algebraic Combinatorics}, 53:327--341, 2021.

\bibitem{bonini2022saturating}
M.~Bonini, M.~Borello, and E.~Byrne.
\newblock Saturating systems and the rank-metric covering radius.
\newblock {\em Journal of Algebraic Combinatorics}, 58:1173--1202, 2023.

\bibitem{borello2023geometric}
M.~Borello and F.~Zullo.
\newblock Geometric dual and sum-rank minimal codes.
\newblock {\em Journal of Combinatorial Designs}, 32(5):238--273, 2024.

\bibitem{brualdiplesswil}
R.~Brualdi, V.~Pless, and R.~Wilson.
\newblock Short codes with a given covering radius.
\newblock {\em IEEE Transactions on Information Theory}, 35(1):99--109, 1989.

\bibitem{bu1980partitions}
T.~Bu.
\newblock Partitions of a vector space.
\newblock {\em Discrete Mathematics}, 31(1):79--83, 1980.

\bibitem{byrne2017covering}
E.~Byrne and A.~Ravagnani.
\newblock Covering radius of matrix codes endowed with the rank metric.
\newblock {\em SIAM Journal on Discrete Mathematics}, 31(2):927--944, 2017.

\bibitem{calderbank1986geometry}
R.~Calderbank and W.~M. Kantor.
\newblock The geometry of two-weight codes.
\newblock {\em Bulletin of the London Mathematical Society}, 18(2):97--122,
  1986.

\bibitem{chen}
H.~Chen.
\newblock New explicit good linear sum-rank-metric codes.
\newblock {\em IEEE Transactions on Information Theory}, 69(10):6303--6313,
  2023.

\bibitem{coveringcodes}
G.~D. Cohen, I.~Honkala, S.~Litsyn, and A.~Lobstein.
\newblock {\em Covering Codes}.
\newblock North-Holland Mathematical Library, 1997.

\bibitem{Nbrega2010MultishotCF}
R.~W. da~N{\'o}brega and B.~F.~U. Filho.
\newblock Multishot codes for network coding using rank-metric codes.
\newblock {\em 2010 Third IEEE International Workshop on Wireless Network
  Coding}, pages 1--6, 2010.

\bibitem{davydov1995constructions}
A.~A. Davydov.
\newblock Constructions and families of covering codes and saturated sets of
  points in projective geometry.
\newblock {\em IEEE Transactions on Information Theory}, 41(6):2071--2080,
  1995.

\bibitem{davydov2011linear}
A.~A. Davydov, M.~Giulietti, S.~Marcugini, and F.~Pambianco.
\newblock Linear nonbinary covering codes and saturating sets in projective
  spaces.
\newblock {\em Advances in Mathematics of Communications}, 5(1):119--147, 2011.

\bibitem{davydov2003saturating}
A.~A. Davydov, S.~Marcugini, and F.~Pambianco.
\newblock On saturating sets in projective spaces.
\newblock {\em Journal of Combinatorial Theory, Series A}, 103(1):1--15, 2003.

\bibitem{davydov2000saturating}
A.~A. Davydov and P.~R. {\"O}sterg{\aa}rd.
\newblock On saturating sets in small projective geometries.
\newblock {\em European Journal of Combinatorics}, 21(5):563--570, 2000.

\bibitem{denaux2021constructing}
L.~Denaux.
\newblock Constructing saturating sets in projective spaces using
  subgeometries.
\newblock {\em Designs, Codes and Cryptography}, pages 1--32, 2021.

\bibitem{dodunekov1998codes}
S.~Dodunekov and J.~Simonis.
\newblock Codes and projective multisets.
\newblock {\em The Electronic Journal of Combinatorics}, 5(1):R37, 1998.

\bibitem{fancsali2014lines}
S.~Fancsali and P.~Sziklai.
\newblock Lines in higgledy-piggledy arrangement.
\newblock {\em the electronic journal of combinatorics}, 21, 2014.

\bibitem{gadouleauphd}
M.~Gadouleau.
\newblock {\em Algebraic codes for random linear network coding}.
\newblock Lehigh University, 2009.

\bibitem{hirschfeld1998projective}
J.~Hirschfeld.
\newblock {\em Projective geometries over finite fields. Oxford Mathematical
  Monographs}.
\newblock Oxford University Press New York, 1998.

\bibitem{lunardon1999normal}
G.~Lunardon.
\newblock Normal spreads.
\newblock {\em Geom. Dedicata}, 75(3):245--261, 1999.

\bibitem{martinez2018skew}
U.~Mart{\'\i}nez-Pe{\~n}as.
\newblock Skew and linearized reed--solomon codes and maximum sum rank distance
  codes over any division ring.
\newblock {\em Journal of Algebra}, 504:587--612, 2018.

\bibitem{M-PKsr}
U.~Martínez-Peñas and F.~R. Kschischang.
\newblock Reliable and secure multishot network coding using linearized
  reed-solomon codes.
\newblock {\em IEEE Transactions on Information Theory}, 65(8):4785--4803,
  2019.

\bibitem{M-PKlr}
U.~Martínez-Peñas and F.~R. Kschischang.
\newblock Universal and dynamic locally repairable codes with maximal
  recoverability via sum-rank codes.
\newblock {\em IEEE Transactions on Information Theory}, 65(12):7790--7805,
  2019.

\bibitem{M-PSK}
U.~Martínez-Peñas, M.~Shehadeh, and F.~R. Kschischang.
\newblock 2022.

\bibitem{neri2023geometry}
A.~Neri, P.~Santonastaso, and F.~Zullo.
\newblock The geometry of one-weight codes in the sum-rank metric.
\newblock {\em Journal of Combinatorial Theory, Series A}, 194:105703, 2023.

\bibitem{ott}
C.~Ott, H.~Liu, and A.~Wachter-Zeh.
\newblock Covering properties of sum-rank metric codes.
\newblock In {\em 2022 58th Annual Allerton Conference on Communication,
  Control, and Computing (Allerton)}, pages 1--7, 2022.

\bibitem{polverino2010linear}
O.~Polverino.
\newblock Linear sets in finite projective spaces.
\newblock {\em Discrete Math.}, 310(22):3096--3107, 2010.

\bibitem{sven+}
S.~Puchinger, J.~Renner, and J.~Rosenkilde.
\newblock Generic decoding in the sum-rank metric.
\newblock {\em IEEE Transactions on Information Theory}, 68(8):5075--5097,
  2022.

\bibitem{randrianarisoa}
T.~H. Randrianarisoa.
\newblock A geometric approach to rank metric codes and a classification of
  constant weight codes.
\newblock {\em Designs, Codes and Cryptography}, 88:1331--1348, 2020.

\bibitem{santonastaso2022subspace}
P.~Santonastaso and F.~Zullo.
\newblock On subspace designs.
\newblock {\em EMS Surveys in Mathematical Sciences}, 2023.

\bibitem{tang2021full}
C.~Tang, Y.~Qiu, Q.~Liao, and Z.~Zhou.
\newblock Full characterization of minimal linear codes as cutting blocking
  sets.
\newblock {\em IEEE Transactions on Information Theory}, 67(6):3690--3700,
  2021.

\end{thebibliography}
	
\end{document}